\documentclass[letterpaper,12pt]{article}
\usepackage{amsmath, mathtools}
\usepackage{amssymb}
\usepackage{amsthm}
\usepackage{graphicx}
\usepackage{verbatim}
\usepackage{color}
\usepackage[margin=1in]{geometry}
\usepackage{tikz}
\usetikzlibrary{shapes,arrows}

\newtheorem{theorem}{Theorem}
\newtheorem{definition}{Definition}

\newtheorem{lemma}{Lemma}
\newtheorem{remark}{Remark}

\def\P{{\mathbb P}}     
\def\E{{\mathbb E}}     

\def\ZZ{\mathbb{Z}}

\definecolor{Red}{rgb}{1,0,0}
\definecolor{Blue}{rgb}{0,0,1}
\def\red{\color{Red}}

\newcommand{\snote}[1]{{\red(S: #1)}}

\title{Impossibility of consistent distance estimation from sequence lengths under the TKF91 model}

\date{\today}
\author{
Wai-Tong Louis Fan\footnote{
Department of Mathematics, Indiana University, Bloomington.} 
\and 
Brandon Legried\footnote{
Department of Mathematics, University of Wisconsin--Madison.} 
\and 
Sebastien Roch\footnote{
Department of Mathematics, University of Wisconsin--Madison.}
}

\begin{document}

\maketitle

\begin{abstract}
    We consider the problem of distance estimation under the TKF91 model of sequence evolution by insertions, deletions and substitutions on a phylogeny. In an asymptotic regime where the expected sequence lengths tend to infinity, we show that no consistent distance estimation is possible from sequence lengths alone. More formally, we establish that the distributions of pairs of sequence lengths at different distances cannot be distinguished with probability going to one.  
\end{abstract}

\section{Introduction}

Phylogeny estimation consists in the inference of an evolutionary tree from extant species data, commonly molecular sequences (e.g.~DNA, amino acid).  A large body of theoretical work exists on the statistical properties of standard reconstruction methods~\cite{Steel:16,Warnow:17}.  Typically in such analyses, one assumes that sequences have evolved on a fixed rooted tree, from a common ancestor sequence to the leaf sequences, according to some Markovian stochastic process.  Often these processes model site substitutions exclusively, with the underlying assumption being that the data has been properly aligned in a pre-processing step.   In contrast, relatively little theoretical work has focused on models of insertions and deletions (indels) together with substitutions, in spite of the fact that such models have been around for some time~\cite{Thorne1991,thorne1992inching}.  See e.g.~\cite{THATTE200658,DaskalakisRoch:13,Allman2015StatisticallyCK,fan2017statistically}.

One extra piece of information available under indel models is the length of the sequence, which itself evolves according to a Markov process on the tree.  The notable work of Thatte~\cite{THATTE200658} shows that leaf sequence lengths \emph{alone} are in fact enough to reconstruct phylogenies, through a distance-based approach.  More specifically, it is shown in~\cite[(27)]{THATTE200658} that under the TKF91 model~\cite{Thorne1991} the expectation of the sequence length $N_v$ at a leaf $v$ conditioned on the sequence length $N_u$ at another leaf $u$ separated from $v$ by an amount of time $t_{uv}$ is
\begin{align}
\label{eq:thatte-m1}
\mathcal{N}_v(t)
= \bar{L} + \left(N_u - \bar{L}\right) 
e^{-\mu t_{uv} (1-\lambda/\mu)}
\end{align}
where $\bar{L} = \frac{\lambda/\mu}{1 - \lambda/\mu}$ is the expected length at stationarity, where $\lambda < \mu$ are the rates of insertion and deletion respectively.  (Full details on the TKF91 model are provided in Section~\ref{sec:defs}.)
Hence we see from~\eqref{eq:thatte-m1} that the full distribution of sequence lengths suffices to recover $\lambda/\mu$  
and all $\mu t_{uv}$'s. 

The tree topology can then be recovered using standard results about the metric properties of phylogenies~\cite{Steel:16}. That is, the tree is identifiable from the sequence lengths under the TKF91 model in the sense that two distinct tree topologies $T_1 \neq T_2$ necessarily produce distinct joint distributions of sequence lengths at the leaves.

It is also suggested in~\cite{THATTE200658}---without a full rigorous proof---that the scheme above could be used to reconstruct phylogenies from a single sample of sequence lengths at the leaves in the limit where $\lambda \nearrow \mu$. The latter asymptotics ensure that the expected sequence length at stationarity $\bar{L}$ diverges, and serves as a proxy for the amount of data growing to infinity.
However, in this short note, we show that no consistent distance estimator exists in this limit. Formally we establish that the distributions of pairs of sequence lengths at different distances cannot be distinguished with probability going to $1$ as $\lambda \nearrow \mu$.  Hence, while the tree is identifiable from the distribution of the sequence lengths at the leaves, one sample of the sequence lengths alone cannot be used in a distance-based approach of the type described above to reconstruct the tree consistently as $\lambda \nearrow \mu$.  On the technical side our proof follows by noting that, under the TKF91 model, the sequence length is (morally) a sum of independent random variables with finite variances, to which we apply a central limit theorem.  One complication is to obtain a limit theorem that is uniform in the parameter $\lambda/\mu$.  We expect that our techniques will be useful to analyze other bioinformatics methods under indel processes, for instance methods based on $k$-mer statistics (see e.g.~\cite{YangZhang:08,Haubold:13}).  Further intuition on our results is provided in Section~\ref{sec:results}. 

\bigskip

{\bf Organization.} The rest of the paper is organized as follows. The TKF91 model is reviewed in Section~\ref{sec:defs}.  Our main result, together with a proof sketch, is stated in Section~\ref{sec:results}.  Details of the proof are provided in Section~\ref{sec:proof}.

\section{Basic definitions}
\label{sec:defs}

In this section, we recall the TKF91 sequence evolution model~\cite{Thorne1991}. To simplify the presentation, we restrict ourselves to a two-state version of the model, as we will only require the underlying sequence-length process.
\begin{definition}[TKF91  model: two-state version]
\label{Def:BinaryIndel}
Consider the following Markov process $\mathcal{I} = \{\mathcal{I}_{t}\}_{t \geq 0}$ on the space $\mathcal{S}$ of binary digit sequences together with an \textbf{immortal link ``$\bullet$''}, that is,
\begin{equation*}\label{S}
		\mathcal{S} := ``\bullet" \otimes \bigcup_{M\geq 1} \{0,1\}^M,
\end{equation*}
where the notation above indicates that all sequences begin with the immortal link.  Positions of a sequence are called \textbf{sites}.  Let $(\nu,\lambda,\mu) \in (0,\infty)^{3}$ and $(\pi_0,\pi_1) \in [0,1]^2$ with $\pi_0 + \pi_1 = 1$ be given parameters.  The continuous-time dynamics are as follows:  If the current state is the sequence $\vec{x} \in \mathcal{S}$, then the following events occur independently:

\begin{itemize}
	\item \emph{Substitution:}  Each site except for the immortal link is substituted independently at rate $\nu > 0$.  When a substitution occurs, the corresponding digit is replaced by $0$ and $1$ with probabilities $\pi_0$ and $\pi_1$, respectively.
	\item \emph{Deletion:}  Each site except for the immortal link is removed independently at rate $\mu$.
	\item \emph{Insertion:}  Each site gives birth to a new digit independently at rate $\lambda$.  When a birth occurs, the new site is added immediately to the right of its parent site.  The newborn site has digit $0$ and $1$ with probabilities $\pi_0$ and $\pi_1$, respectively.
\end{itemize}
\end{definition}
\noindent This indel process is time-reversible with respect to the measure $\Pi$ given by
\begin{equation*}\label{Pi}
	\Pi(\vec{x})=
	\left(1-\frac{\lambda}{\mu}\right) 
	\left(\frac{\lambda}{\mu}\right)^M\prod_{i=1}^M\pi_{x_i} 
\end{equation*}
for each $\vec{x}=(x_1,x_2,\cdots,x_M)\in \{0,1\}^M$ where $M\geq 1$, and $\Pi(``\bullet") = \left(1-\frac{\lambda}{\mu}\right)$. We assume further that $\lambda < \mu$.  In that case, $\Pi$ is the stationary distribution of  $\mathcal{I}$.

We will be concerned with the underlying sequence-length process.
\begin{definition}[Sequence length]
The \textbf{length} of a sequence $\vec{x} = (\bullet,x_1,...,x_M)$ is defined as the number of sites except for the immortal link and is denoted by $|\vec{x}| = M$.
\end{definition}
 \noindent Under $\Pi$, the sequence-length process $|\mathcal{I}|$ is stationary and is geometrically distributed. 
Specifically the stationary distribution of the length process $|\mathcal{I}|$  is
\begin{equation}\label{gamma_M}
\gamma^{(\lambda)}_M:= \left(1-\frac{\lambda}{\mu}\right)\left(\frac{\lambda}{\mu}\right)^M,\quad M\in\ZZ_+.
\end{equation}

We are interested in this process on a rooted tree $T$.  Denote the index set by $\Gamma_{T}$.  The root vertex $\rho$ is assigned a state $\mathcal{I}_{\rho} \in \mathcal{S}$, drawn from stationary distribution on $\mathcal{S}$.  This state is then evolved down the tree according to the following recursive process.  Moving away from the root, along each edge $e = (u,v) \in E$, conditionally on the state $\mathcal{I}_{u}$, we run the indel process for a time $\ell_{(u,v)}$.  Denote by $\mathcal{I}_{t}$ the resulting state at $t \in e$.  
Then the full process is denoted by $\{\mathcal{I}_{t}\}_{t \in \Gamma_{T}}$.  
In particular, the set of leaf states is $\mathcal{I}_{\partial T} = \{\mathcal{I}_{v}:v \in \partial T\}$.

\bigskip

{\bf Setting. }
Throughout this paper, we let $\mathbb{P}_{\vec{x}}$ be the probability measure when the root state is $\vec{x}$.  If the root state is chosen according to a distribution $\nu$, then we denote the probability measure by $\mathbb{P}_{\nu}$.  We also denote  by $\mathbb{P}_{M}$ the conditional probability measure for the event that the root state has length $M$.

For our purposes, it will suffice to focus on the space $\mathcal{T}_{2}$ of star trees with two leaves that  have the same finite distance $h\in (0,\infty)$ from the root and are labeled as $\{1,2\}$. This distance $h$ is the height of the tree.
The indel process on a tree $T\in \mathcal{T}_{2}$ reduces to a pair of indel processes $(\mathcal{I}_{t}^1,\mathcal{I}_{t}^2)_{t\geq 0}$ that are independent upon conditioning on the root state $\mathcal{I}_{\rho}=\mathcal{I}_{0}^1=\mathcal{I}_{0}^2$.
We always assume the root state is chosen according to the equilibrium distribution $\Pi$. So the distribution of $(\mathcal{I}_{0}^1,\mathcal{I}_{0}^2)\in \mathcal{S} \times \mathcal{S}$ is 
\begin{equation*}\label{hatnu0}
\widehat{\nu}_0(\vec{x},\vec{y}) = \begin{cases}
\Pi(\vec{x}) & \text{ if } \vec{x} = \vec{y}, \\
0 & \textnormal{otherwise.}
\end{cases}
\end{equation*}

\section{Main result}
\label{sec:results}

Our main theorem is an impossibility result: the distributions of pairs of sequence lengths at different distances cannot be distinguished with probability going to $1$ as $\lambda \nearrow \mu$. Following~\cite{THATTE200658}, we consider the asymptotic regime where $\lambda \nearrow \mu$, which implies that the expected sequence length at stationarity tends to $+\infty$.
Recall that the total variation distance between two probability measures
$\tau_1$ and $\tau_2$ on a countable measure space $E$ is 
\begin{align*}
\|\tau_1 - \tau_2\|_{TV}
= \frac{1}{2} \sum_{\sigma \in E}
\left|\tau_1(\sigma)-\tau_2(\sigma)\right|.
\end{align*}
\begin{theorem}[Impossibility of distance estimation from sequence lengths]
\label{T:Length}
	Let $T^{1}$ and $T^{2}$ be two trees in $\mathcal{T}_2$ with heights $h_1 > h_2>0$ respectively.   For $i\in \{1,2\}$, we consider a TKF91 process on tree $T^i$ and let $\vec{N}^{(i)} = (N_{1}^{(i)},N_{2}^{(i)})\in\mathbb{Z}_{+}^2$ be the pair of sequence lengths at the leaves $\partial T^i$.
	Let
	\begin{equation*}
	    \mathcal{L}_{i} = \mathbb{P}_{\Pi}(\vec{N}^{(i)} \in \cdot)
	\end{equation*}
	be the distribution of $\vec{N}^{(i)}$ under $\mathbb{P}_{\Pi}$.  Then  for any fixed deletion rate $\mu \in (0,\infty)$,
	\begin{equation}
	\label{E:Length}
	   \limsup_{\lambda \nearrow \mu}\|\mathcal{L}_{1} - \mathcal{L}_{2}\|_{TV} < 1.
	\end{equation}
\end{theorem}

Recall that the total variation distance can be written as
\begin{align*}
\|\tau_1 - \tau_2\|_{TV}
= \sup_{A \subseteq E} \left|\tau_1(A) - \tau_2(A)\right|.
\end{align*}
So~\eqref{E:Length} implies that there is no test that can distinguish between $\mathcal{L}_{1}$ and $\mathcal{L}_{2}$ with probability going to $1$ as $\lambda \nearrow \mu$.

\bigskip

{\bf Proof idea. }
We first give a heuristic argument that underlies our formal proof.  Without loss of generality, assume that the deletion rate is $\mu = 1$.
The stationary length $M$ at the root is geometric with mean and standard deviation both of order $1/(1-\lambda)$. So we can think of the root length as roughly $M \approx C/(1-\lambda)$ with significant probability. Ignoring the small effect of the immortal link and conditioning on $M$, the lengths at the leaves are sums of independent random variables, specifically the progenies of the $M$ mortal links of the root. The mean and variance of these variables can be computed explicitly from continuous-time Markov chain theory (see~\eqref{betasigma} below; see also~\cite[(27), (31)]{THATTE200658}). As $\lambda \nearrow 1$, the difference in expectation between heights $h_1$ and $h_2$ turns out to be
\begin{align}
\label{eq:sketch1}
    M e^{-(1-\lambda)h_1} - M e^{-(1-\lambda)h_2}
    \approx \frac{C}{1-\lambda}[-(1-\lambda)h_1 + (1-\lambda) h_2] 
    \approx C(h_2 - h_1),
\end{align}
while the variance is of order
\begin{align}
\label{eq:sketch2}
    M \frac{e^{-(1-\lambda)h_i} (1- e^{-(1-\lambda)h_i})}{1-\lambda}
    \approx \frac{C}{1-\lambda} \frac{ (1-\lambda)h_i}{1-\lambda}
    \approx \frac{C h_i}{1-\lambda}.
\end{align}
The key observation is that the variance~\eqref{eq:sketch2} is $\gg$ than the square of the expectation difference~\eqref{eq:sketch1}. Hence, by the central limit theorem, one can expect significant overlap between the length distributions under $h_1$ and $h_2$, making them hard to distinguish even as $\lambda \nearrow 1$. We formalize this argument next. 

We observe that \eqref{E:Length} is  equivalent to 
\begin{equation}\label{S:c0}
  \liminf_{\lambda \nearrow 1}\sum_{\vec{y} \in \mathbb{Z}_{+}^2}\mathbb{P}_{\Pi}(\vec{N}^{(1)} = \vec{y}) \wedge \mathbb{P}_{\Pi}(\vec{N}^{(2)} = \vec{y}) > 0.  
\end{equation}
Indeed the total variation distance between two probability measures
$\tau_1$ and $\tau_2$ on a countable space $E$ can also be written as
\begin{align*}
\|\tau_1 - \tau_2\|_{TV}
= 1 - \sum_{\sigma \in E}\tau_1(\sigma)\land\tau_2(\sigma). 
\end{align*}
The rest of the proof is to establish \eqref{S:c0}.  It involves a series of steps:
\begin{enumerate}
    \item We first reduce the problem to a sum of independent random variables by conditioning on the root sequence length and ignoring the immortal link. In particular, we use the fact that there is a fairly uniform probability that $M$ is in an interval of size $1/(1-\lambda)$ around $1$. And we remove the effect of the immortal link by conditioning on its having no descendant, an event of positive probability.
    
    \item The central limit theorem (CLT) implies that there is a significant overlap between the two sums. More precisely, we need a local CLT (see e.g.~\cite{Durrett:10}) to derive the sort of pointwise lower bound needed in \eqref{S:c0}. However the bound we require must be uniform in $\lambda$ and we did not find in the literature a result of quite this form. Instead, we use an argument based on the Berry-Ess\'een theorem (again see e.g.~\cite{Durrett:10}). We first establish overlap over $\Omega(\sqrt{M})$ constant size intervals for the sum of the first $M-1$ mortal links, and then we use the final mortal link to match the probabilities on common point values under heights $h_1$ and $h_2$.
    
    \item Finally we bound the sum in \eqref{S:c0}. 
\end{enumerate}

\section{Proof}
\label{sec:proof}

In this section
we give the details of the proof of Theorem \ref{T:Length}.  We follow the steps described in the previous section. 

\bigskip

\noindent
{\bf Step 1. Reducing the problem to a sum of independent random variables. } We first show that  $\mathbb{P}_{\Pi}$ in \eqref{S:c0} can be replaced by $\P_M$ where $M$ is of the order of the expected sequence length $1/(1-\lambda)$ under $\Pi$. That is, we condition on the length of the ancestral sequence. After that we further ignore the progenies of the immortal link so that each leave sequence consists of i.i.d.~progenies of the $M$ sites in the ancestral  sequence.  These two simplifications are achieved in \eqref{S:eq0} and \eqref{S:eq2} below respectively. 

Precisely, for any $\lambda_* \in (0,1)$ and $0 < c_1 < 1 < c_2 < +\infty$, using~\eqref{gamma_M}
\begin{align}
&\liminf_{\lambda \nearrow 1}\sum_{\vec{y} \in \mathbb{Z}_{+}^2}\mathbb{P}_{\Pi}(\vec{N}^{(1)} = \vec{y}) \wedge \mathbb{P}_{\Pi}(\vec{N}^{(2)} = \vec{y}) \notag\\
&\geq  \inf_{\lambda \in (\lambda_*,1)}\sum_{\vec{y} \in \mathbb{Z}_{+}^2}\mathbb{P}_{\Pi}(\vec{N}^{(1)} = \vec{y}) \wedge \mathbb{P}_{\Pi}(\vec{N}^{(2)} = \vec{y}) \notag\\
&= \inf_{\lambda \in (\lambda_*,1)}\sum_{\vec{y} \in \mathbb{Z}_{+}^2}
\left[\sum_{M \in \mathbb{Z}_+} \gamma^{(\lambda)}_M \,\mathbb{P}_{M}(\vec{N}^{(1)} = \vec{y}) \right]\wedge 
\left[\sum_{M \in \mathbb{Z}_+} \gamma^{(\lambda)}_M \,\mathbb{P}_{M}(\vec{N}^{(2)} = \vec{y})\right] \notag\\
&= \inf_{\lambda \in (\lambda_*,1)}
\sum_{\vec{y} \in \mathbb{Z}_{+}^2}
\sum_{M \in \mathbb{Z}_+} (1-\lambda) \lambda^M \left[
\mathbb{P}_{M}(\vec{N}^{(1)} = \vec{y}) \wedge \mathbb{P}_{M}(\vec{N}^{(2)}= \vec{y})\right] \notag\\
&\geq \inf_{\lambda \in (\lambda_*,1)}
\sum_{M \in \left[\frac{c_1}{1-\lambda},\frac{c_2}{1-\lambda}\right]} (1-\lambda) \lambda^M
\sum_{\vec{y} \in \mathbb{Z}_{+}^2}
\mathbb{P}_{M}(\vec{N}^{(1)} = \vec{y}) \wedge \mathbb{P}_{M}(\vec{N}^{(2)}= \vec{y}) \notag\\
&\geq c_3 (c_2-c_1) 
\inf_{\lambda \in (\lambda_*,1)}
\inf_{M \in \left[\frac{c_1}{1-\lambda},\frac{c_2}{1-\lambda}\right]}
\sum_{\vec{y} \in \mathbb{Z}_{+}^2}
\mathbb{P}_{M}(\vec{N}^{(1)} = \vec{y}) \wedge \mathbb{P}_{M}(\vec{N}^{(2)}= \vec{y}),\label{S:eq0}
\end{align}
where $c_3$ is a lower bound on 
$\lambda^{M}$ for $M \in \left[\frac{c_1}{1-\lambda},\frac{c_2}{1-\lambda}\right]$ and $\lambda \in (\lambda_*,1)$.

\medskip

Let $\mathcal{Z}_0$ be the event that
the immortal link of the root sequence produces no mortal link in either leaf sequences. Let $\mathbb{P}_{M,\bullet}$ be
the probability conditioned on that event,
and  $c_4$ be a lower bound on the
probability of $\mathcal{Z}_0$ uniform
in $\lambda \in (\lambda_*,1)$. 
Under $\mathbb{P}_{M,\bullet}$, the two components of $\vec{N}^{(1)}$ are conditionally independent and each is a sum of $M$ i.i.d.~random variables corresponding to the progenies of mortal links. 
Hence \eqref{S:eq0} is at least
\begin{align}
&c_4 c_3 (c_2-c_1) 
\inf_{\lambda \in (\lambda_*,1)}
\inf_{M \in \left[\frac{c_1}{1-\lambda},\frac{c_2}{1-\lambda}\right]}
\sum_{\vec{y} \in \mathbb{Z}_{+}^2}
\mathbb{P}_{M,\bullet}(\vec{N}^{(1)} = \vec{y}) \wedge \mathbb{P}_{M,\bullet}(\vec{N}^{(2)}= \vec{y}) \nonumber\\
&\geq c_4 c_3 (c_2-c_1) 
\inf_{\lambda \in (\lambda_*,1)}
\inf_{M \in \left[\frac{c_1}{1-\lambda},\frac{c_2}{1-\lambda}\right]}
\sum_{\vec{y} \in \mathbb{Z}_{+}^2}
\left[p_{M,y_1}^{(\lambda)}(h_1) \,p_{M,y_2}^{(\lambda)}(h_1) \right] \wedge \left[p_{M,y_1}^{(\lambda)}(h_2) \,p_{M,y_2}^{(\lambda)}(h_2) \right].\label{S:eq2}
\end{align}
where  we
let $p_{y_1,y_2}^{(\lambda)}(t) = \mathbb{P}_{y_1,\bullet}(|\mathcal{I}_{t}| = y_2)$ be the transition probability of the length process \textit{without} the immortal link.

The sum in \eqref{S:eq2} leads us to study the overlap between the probability distributions $p_{M, \,\cdot}^{(\lambda)}(t):=\{p_{M, j}^{(\lambda)}(t)\}_{j\in\ZZ_+}$ for $t=h_1, h_2$ and $M\in  \left[\frac{c_1}{1-\lambda},\frac{c_2}{1-\lambda}\right]$.  The central limit theorem is what we need. However, because of our need for a bound that is uniform in $\lambda$,
we shall apply the Berry-Ess\'een theorem.
Specifically, we apply the latter 
bound to the progenies of the first $M-1$ mortal links of the root sequence. The idea is to show that $\Omega(\sqrt{M})$ summands in \eqref{S:eq2} have value $\Omega(1/\sqrt{M})$, for each
of $h_1$ and $h_2$ separately, and then use
the last mortal link to ``match'' all these values between $h_1$ and $h_2$.

\bigskip

\noindent
{\bf Step 2a. Establishing a uniform bound for $p_{M-1, \cdot}^{(\lambda)}(t)$. }
Note that $p_{M, \cdot}^{(\lambda)}(t)$ is the distribution of $S_{M}(t):=\sum_{i=1}^ML_t^i$, where $\{L^i_t\}_{i\geq 1}$ are i.i.d. random variables having the distribution of  the progeny length of a single mortal link at time $t>0$.  

Let the mean and the variance of $L^i_t$ be
\begin{equation}
\label{Def:betasigma}
\beta:=\beta(\lambda,t):=\E[L^i_t]\quad\text{and}\quad\sigma^2:=\sigma^2(\lambda,t):=\E|L^i_{t}-\beta|^2.
\end{equation}
As is expected,
{\it the distribution $p_{M, \,\cdot}^{(\lambda)}(t)$ is approximately Gaussian with mean $\beta M$ and variance $\sigma^2 M$}. 
We quantify this statement in the bound \eqref{BerryEsseen} below, after proving some moment bounds.
\begin{lemma}\label{L:moments}
Let $\beta(\lambda,t)$ and $\sigma(\lambda,t)$ be the mean and the standard deviation of $L^i_t$ defined in \eqref{Def:betasigma} and
consider the absolute third moment $\rho(\lambda,t) := E|L^i_{t}-\beta|^3$. For any $t\in (0,\infty)$,
\begin{equation}
\label{betasigma}
\beta(\lambda,t)=e^{-(1-\lambda)t}\quad\text{and}\quad \sigma^2(\lambda,t)=\frac{1+\lambda}{1-\lambda} e^{-(1-\lambda)t} (1- e^{-(1-\lambda)t}).
\end{equation}
Furthermore,
\[
0<\inf_{\lambda\in [\lambda_*,1]}\sigma(\lambda,t) < \sup_{\lambda\in [\lambda_*,1]}\sigma(\lambda,t) <\infty \quad\text{and}\quad 
\sup_{\lambda\in [\lambda_*,1]}\rho(\lambda,t) <\infty.
\]
\end{lemma}
\begin{proof}
For~\eqref{betasigma}, see e.g.~\cite[(3), (4)]{DaskalakisRoch:13}.

Moreover, from \cite[(8)--(10)]{Thorne1991}, 
the probability that a normal link as $n$ 
 descendants including itself is 
	\begin{equation*}
	\P(L^i_t=n)=
	\begin{cases}
     (1-\eta(\lambda, t))(1-\lambda\eta(t))[\lambda\eta(\lambda, t)]^{n-1} \qquad &\text{for }n\ge 1 \\
	 \eta(\lambda, t) \qquad &\text{for }n=0
	\end{cases},
	\end{equation*}
where 
$\eta(\lambda, t) =
\frac{1- e^{-(1-\lambda)t}}{1-\lambda e^{-(1-\lambda)t}}$. It can be seen from L'Hospital's rule that $\eta(\lambda, t)$ is continuous as a function of $\lambda$ around $1$ and that $\eta(\lambda, t) = \frac{t}{1+t} + O(|1-\lambda|)$ as $\lambda \to 1$.  From this explicit formula for the probability mass function of $L^i_t$, which we note is a geometric sequence, it follows that all moments of $L^i_t$ are  bounded from above uniformly in $\lambda\in [\lambda_*,1]$. 

To show that the variance is bounded from below uniformly in $\lambda\in [\lambda_*,1]$, we note (again using L'Hospital's rule) that $
\sigma^2(\lambda,t)$ is continuous in $\lambda$ around $1$, strictly positive and tends to $2t$ as $\lambda \to 1$. 
Hence the variance is bounded from below, uniformly in $\lambda\in [\lambda_*,1]$
\end{proof}
Let $F_{M}^{(\lambda)}(t)$ be the cumulative distribution function (CDF) of the probability distribution $p_{M, \cdot}^{(\lambda)}(t)$. That is, 
\[F_{M}^{(\lambda)}(t)(x)=\sum_{j:j\leq x}p_{M, j}^{(\lambda)}(t)=\P(S_{M}(t)\leq x).\]
\begin{lemma}[Uniform bound for $p_{M-1, \cdot}^{(\lambda)}(t)$]
For each $t>0$, there exists a constant $C>0$ such that
\begin{equation}\label{BerryEsseen}
    \sup_{\lambda\in[\lambda_*,1]}\sup_{x \in \mathbb{R}}\Big|F_{M}^{(\lambda)}(t)\Big(M\beta(\lambda,t)\,+\,x\,\sigma(\lambda,t)\,\sqrt{M}\Big) - \mathcal{N}(x)\Big| \leq \frac{C}{\sqrt{M}},
\end{equation}
for all $M\in Z_+$,
where $\mathcal{N}$ is the CDF of the standard normal distribution.
\end{lemma}

\begin{proof}
Since $\beta(\lambda,t), \sigma^2(\lambda,t), \rho(\lambda,t) \in (0,\infty)$, the Berry-Ess\'een theorem (as stated e.g.~in~\cite{Durrett:10}) applies and asserts that  
\[ \sup_{x \in \mathbb{R}}\left|\P\left(\frac{S_{M-1} - (M-1)\beta(\lambda,t)}{\sigma(\lambda,t)\sqrt{M-1}}\leq x\right) \,-\, \mathcal{N}(x) \right| \leq \frac{3\rho(\lambda,t)}{\sigma^3(\lambda,t) \sqrt{M-1}}\] 
for all $\lambda\in[0,1]$. By Lemma \ref{L:moments}, 
for each $t>0$, the right hand side is bounded from above uniformly for $\lambda\in[\lambda_*,1)$.
\end{proof}

\bigskip

\noindent
{\bf Step 2b. Controlling the overlap of $p_{M-1, \cdot}^{(\lambda)}(h_1)$ and $p_{M-1, \cdot}^{(\lambda)}(h_2)$ in \eqref{S:eq2}. } To quantify the overlap between $p_{M-1, \cdot}^{(\lambda)}(h_1)$ and $p_{M-1, \cdot}^{(\lambda)}(h_2)$, we first compare their expectations.
From the formula of $\beta$ in \eqref{betasigma}, we have
\begin{equation*}\label{mean_diff}
    \beta(\lambda,h_1) - \beta(\lambda,h_2) 
\leq (1-\lambda)(h_1 - h_2)
\end{equation*}
and so,
for $M \in \left[\frac{c_1}{1-\lambda},\frac{c_2}{1-\lambda}\right]$, the means of  $S_{M-1}$ for $h_1$ and $h_2$ are close in the sense that
\begin{equation}\label{mean_diff2}
\beta(\lambda,h_1)(M-1) - \beta(\lambda,h_2)(M-1)
\leq c_6 
\end{equation}
for some $c_{6} > 0$ not depending on $\lambda$. 

Now consider the interval with length roughly the standard deviation  and centered at around one of the means, $\beta(\lambda,h_1)(M-1)$. Then  consider an equi-partition of this interval  into roughly $\sqrt{M-1}$ many pieces of constant length. Precisely, we write $\beta_1:=\beta(\lambda,h_1)$ and $\sigma_1:=\sigma(\lambda,h_1)$ for simplicity. Then
for an arbitrary constant $K > 0$,
\[
 \Big(\beta_1(M-1)
 -\sigma_1\sqrt{M-1},\; \beta_1(M-1)+\sigma_1\sqrt{M-1} \Big)
 =\bigcup_{r\in \Lambda^M_K}\mathcal{J}^M_r(c)
\]
where the sub-intervals 
$$
\mathcal{J}^M_r(K):= \big(\beta_1(M-1)
+ r\sigma_1 K ,\; \beta_1(M-1)
+ (r+1) \sigma_1 K  \big)
$$
have constant width $\sigma_1 K$
for $r \in\Lambda^M_K$,  and
\[
\Lambda^M_K:=\left\{-\sigma_1 \sqrt{M-1},\, -\sigma_1 (\sqrt{M-1}-K)\,\ldots, 0,\,\sigma_1 K,\,\ldots, \sigma_1(\sqrt{M-1}-K)\right\}.
\]

 \begin{figure}[t] \centering
	\begin{tikzpicture}[>=latex,scale=0.6]
	\draw[->] (0,0)--(10.5,0);
	\draw[->] (0,0)--(0,8.5);
	\draw(11.2,0)node{\large$\mathbb{Z}_+$};
	\draw(0, 9)node{\large$\mathbb{Z}_+$};
	\draw(6.5,-0.4)node{$M-1$};
	

	\draw[line width=1pt](0,0)--(10.5, 8.5);
	\draw[dashed](0,0)--(10.5, 7.5);

\draw[line width=1pt](6.5,2)--(6.5,8.6); 
\shade[ball color=black](6.5,5.3)circle(1.2mm);

\shade[ball color=black](6.5,2)circle(1.2mm);
\shade[ball color=black](6.5, 3.65)circle(1.2mm);

\shade[ball color=black](6.5,8.6)circle(1.2mm);
\shade[ball color=black](6.5,6.95)circle(1.2mm);
	\end{tikzpicture}
	\caption{\small The solid straight line $j= \beta_1 M$ has slope $\beta_1$ and the dotted line $j= \beta_2 M$ has slope $\beta_2$ where $\beta_i=\beta(\lambda,h_i)$ for $i=1,2$. 
	The vertical line has length $2\sigma_1\sqrt{M-1}$ where $\sigma_1=\sigma(\lambda,h_1)$ and represents the union of sub-intervals $\bigcup_{r\in \Lambda^M_K}\mathcal{J}^M_r(c)$. 
	Lemma \ref{L:c_7c_8} says that for each $M \in \left[\frac{c_1}{1-\lambda},\frac{c_2}{1-\lambda}\right]$, 
	both probability measures  $p_{M-1, \,\cdot}^{(\lambda)}(h_1)$ and $p_{M-1,\,\cdot}^{(\lambda)}(h_2)$ have mass at least $c_8/\sqrt{M-1}$ on  $\mathcal{J}^M_r(c)$, uniformly for all  $r\in \Lambda^M_K$ and $\lambda\in [\lambda_*,1)$.
	}
	\label{fig:overlap} 
\end{figure}
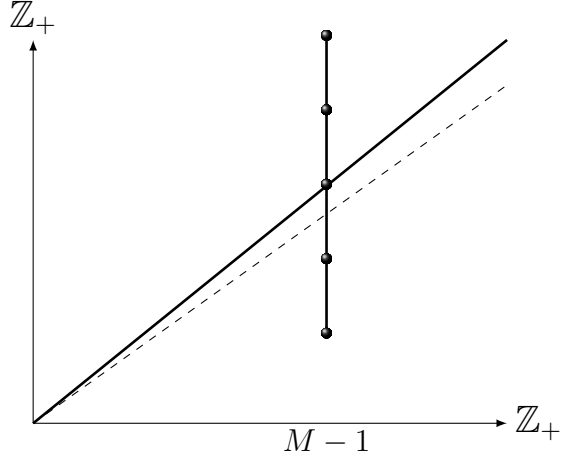

Lemma \ref{L:c_7c_8} below says that there exists a constants $c=c_7$ and $K$ large enough
(depending on $c_5$ and $c_6$ but not on $\lambda$) such that
each of these intervals contains
mass at least $\frac{c_8}{\sqrt{M-1}}$
under both probability distributions $p_{M-1, \,\cdot}^{(\lambda)}(h_1)$
and $p_{M-1, \,\cdot}^{(\lambda)}(h_2)$. See Figure \ref{fig:overlap}. Write $p_{M-1, \,A}^{(\lambda)}(t)=\sum_{j\in A}p_{M-1, \,j}^{(\lambda)}(t)$ for simplicity.
\begin{lemma}\label{L:c_7c_8}
There exist positive constants  $c_7,c_8$ such that, with $\mathcal{J}_{r}=\mathcal{J}^M_{r}(c_7)$ and $\Lambda^M=\Lambda^M_{c_7}$,
\[p_{M-1, \,\mathcal{J}_{r}}^{(\lambda)}(h_1) \wedge p_{M-1, \,\mathcal{J}_{r}}^{(\lambda)}(h_2) \geq \frac{c_8}{\sqrt{M-1}}\]
for all $r \in\Lambda^M$,  $M \in \left[\frac{c_1}{1-\lambda},\frac{c_2}{1-\lambda}\right]$ and $\lambda\in[\lambda_*,1)$.
\end{lemma}

\begin{proof}
The Berry-Ess\'een theorem \eqref{BerryEsseen} implies that
\[\sup_{r\in\Lambda^M}\left|p_{M-1, \,\mathcal{J}_{r}}^{(\lambda)}(h_1) - \int_{\widetilde{\mathcal{J}_{r}}}\frac{1}{\sqrt{2\pi}}e^{-\frac{x^2}{2}}\,dx \right| \leq \frac{6\rho}{\sigma^3 \sqrt{M-1}}\]
for all $\lambda\in[\lambda_*,1]$ and $M\geq 2$, 
where
\[
\widetilde{\mathcal{J}_{r}}:=\left(\frac{rK}{\sqrt{M-1}},\,\frac{(r+1)K}{\sqrt{M-1}}  \right).
\]
Then $\{\widetilde{\mathcal{J}_{r}}\}_{r\in\Lambda^M}$ is roughly an equi-partition of the interval $(-1,1)$ into $\frac{2\sqrt{M-1}}{K}$ sub-intervals of length $\frac{K}{\sqrt{M-1}}$. Furthermore,
$$\int_{\widetilde{\mathcal{J}_{r}}}\frac{1}{\sqrt{2\pi}}e^{-\frac{x^2}{2}}\,dx \geq \frac{K}{\sqrt{M-1}}\,\frac{e^{-1/2}}{\sqrt{2\pi}}.$$

Pick $K$ large enough (call it $c_7$), we obtain from the first display in this proof that

\[  \inf_{r \in\Lambda^M,\,\lambda\in[\lambda_*,1)} p_{M-1, \,\mathcal{J}_{r}}^{(\lambda)}(h_1) \geq \frac{c}{\sqrt{M-1}}\] 
for some constant $c > 0$ that does not depend on $M$. By the same argument and using \eqref{mean_diff2}, we have
\[ \inf_{r \in\Lambda^M,\,\lambda\in[\lambda_*,1)} p_{M-1, \,\mathcal{J}_{r}}^{(\lambda)}(h_2) \geq \frac{c'}{\sqrt{M-1}}\] 
for some constant $c'>0$ that does not depend on $M$,
even though $\mathcal{J}_{r}$ is constructed using $h_1$. The proof is complete by taking $c_8=\min\{c,c'\}$.
\end{proof}

\bigskip

\noindent
{\bf Step 2c. Matching $p_{M, \cdot}^{(\lambda)}(h_1)$ and $p_{M, \cdot}^{(\lambda)}(h_2)$ by the last mortal link. } Lemma \ref{L:c_7c_8} establishes overlap of $p_{M-1, \cdot}^{(\lambda)}(h_1)$ and $p_{M-1, \cdot}^{(\lambda)}(h_2)$ over constant size intervals. The next lemma uses the final mortal link to establish overlap of $p_{M, \cdot}^{(\lambda)}(h_1)$ and $p_{M, \cdot}^{(\lambda)}(h_2)$ over specific values. 
\begin{lemma}\label{L:c_9}
	There exists a positive constant $c_9$ such that 
	\[  \inf_{j_{r+1}^{\ast} \in \mathcal{J}_{r+1}\cap \mathbb{Z}_+ }\,p_{M,j_{r+1}^{\ast}}^{(\lambda)}(h_1) \wedge p_{M,j_{r+1}^{\ast}}^{(\lambda)}(h_2) > \frac{c_8 c_9}{c_7\sqrt{M-1}}.\]
	for all  $r \in\Lambda^M$,  $M \in \left[\frac{c_1}{1-\lambda},\frac{c_2}{1-\lambda}\right]$ and $\lambda\in[\lambda_*,1)$.
\end{lemma}
\begin{proof}
By Lemma \ref{L:c_7c_8}, 
$\mathcal{J}_r$ contains
at least one integer, say $j_r^{(1)}$, with mass at least
$\frac{c_8}{c_7 \sqrt{M-1}}$
under the probability measure  $p_{M-1, \,\cdot}^{(\lambda)}(h_1)$. This is because there are $c_7$ integers in $\mathcal{J}_r$.
Similarly, there exists $j_{r}^{(2)}$ with mass at least $\frac{c_8}{c_7 \sqrt{M-1}}$ under $m_{F_{M-1}^{(\lambda)}(h_2)}$. Hence
\[ p_{M-1, \,j_{r}^{(1)}}^{(\lambda)}(h_1) \wedge p_{M-1, \,j_{r}^{(2)}}^{(\lambda)}(h_2)\geq \frac{c_8}{c_7 \sqrt{M-1}}.\]

Let $j_{r+1}^*$ be an arbitrary 
integer in $\mathcal{J}_{r+1}$. 
The progeny of the $M$-th mortal link
has positive probability, say $c_9$, over integers in $[0,2 c_7]$,  uniformly over $\lambda\in [0,1]$. It follows that
\begin{equation*}\label{MatchLastTerm}
p_{M,j_{r+1}^*}^{(\lambda)}(h_1)=\sum_{k=0}^{j_{r+1}^*}p_{M-1,k}^{(\lambda)}(h_1)\,p_{1,j_{r+1}^*-k}^{(\lambda)}(h_1) > p_{M-1,j_r^{(1)}}^{(\lambda)}(h_1)\,p_{1,j_{r+1}^*-j_r^{(1)}}^{(\lambda)}(h_1) \geq \frac{c_8 c_9}{c_7 \sqrt{M-1}}
\end{equation*}
and similar for $h_2$. The proof is complete.
\end{proof}

\noindent
{\bf Step 3. Putting everything together. }
Lemma \ref{L:c_9} implies the sum in~\eqref{S:eq2} is at least
a positive constant, uniformly in 
$M \in \left[\frac{c_1}{1-\lambda},\frac{c_2}{1-\lambda}\right]$
and $\lambda \in (\lambda_*,1)$, because that sum is 
\begin{align*}
&\,\sum_{\vec{y} \in \mathbb{Z}_{+}^2}\left[p_{M,y_1}^{(\lambda)}(h_1) \,p_{M,y_2}^{(\lambda)}(h_1) \right] \wedge \left[p_{M,y_1}^{(\lambda)}(h_2) \,p_{M,y_2}^{(\lambda)}(h_2) \right]\\
\geq\,&  \sum_{y_1\in \cup_{r\in \Lambda^M}\mathcal{J}_{r+1}\cap \mathbb{Z}_+,\;y_2\in \cup_{r\in \Lambda^M}\mathcal{J}_{r+1}\cap \mathbb{Z}_+ }\left[p_{M,y_1}^{(\lambda)}(h_1)  \wedge p_{M,y_1}^{(\lambda)}(h_2) \right]\,\cdot \,\left[p_{M,y_2}^{(\lambda)}(h_1)  \wedge p_{M,y_2}^{(\lambda)}(h_2) \right]\\
\geq\,& \left(\frac{c_8 c_9}{c_7 \sqrt{M-1}}\right)^2\,\Big|\{y_1\in \cup_{r\in \Lambda^M}\mathcal{J}_{r+1}\cap \mathbb{Z}_+,\;y_2\in \cup_{r\in \Lambda^M}\mathcal{J}_{r+1}\cap \mathbb{Z}_+\}\Big|\\
\geq\,&  \left(\frac{c_8 c_9}{c_7}\right)^2.
\end{align*}
The proof of \eqref{S:c0} and hence that of Theorem  \ref{T:Length} are complete.

\section*{Acknowledgments}

SR was supported by NSF grants DMS-1614242, CCF-1740707 (TRIPODS), DMS-1902892, and DMS-1916378, as well as a Simons Fellowship and a Vilas Associates Award. BL was supported by DMS-1614242, CCF-1740707 (TRIPODS), DMS-1902892 (to SR). WTF was supported by NSF grant DMS-1614242 (to SR) and DMS-1855417.

\bibliographystyle{alpha}
\bibliography{lengthbib}{}

\end{document}